\documentclass{amsart}
\usepackage[margin=1.4in]{geometry}
\usepackage{amsmath,amsthm,amssymb,comment,enumerate}
\usepackage{graphicx}
\usepackage{color}

\theoremstyle{plain}
\newtheorem{thm}{Theorem}[section]
\newtheorem{cor}[thm]{Corollary}
\newtheorem{lem}[thm]{Lemma}
\newtheorem{prop}[thm]{Proposition}
\newtheorem{conj}[thm]{Conjecture}

\theoremstyle{definition}

\theoremstyle{remark}
\newtheorem{rem}[thm]{Remark}

\numberwithin{equation}{section}
\begin{document}
\title{On the spectral asymptotics for the buckling problem}%
\author{Davide Buoso}%
\author{Paolo Luzzini}
\author{Luigi Provenzano}%
\author{Joachim Stubbe}%

\address{Davide Buoso, Dipartimento di Scienze e Innovazione Tecnologica (DiSIT), Università degli Studi del Piemonte Orientale ``A. Avogadro'', Viale Teresa Michel 11, 15121 Alessandria (ITALY). E-mail: {\tt davide.buoso@uniupo.it}}%
\address{Paolo Luzzini, EPFL, SB MATH SCI-SB-JS, Station 8, CH-1015 Lausanne, Switzerland. E-mail: {\tt paolo.luzzini@epfl.ch}}%
\address{Luigi Provenzano, Dipartimento di Scienze di Base e Applicate per l'Ingegneria, Sapienza Universit\`a di Roma, Via Antonio Scarpa 16, 00161 Roma, Italy. E-mail: {\tt luigi.provenzano@uniroma1.it} }
\address{Joachim Stubbe, EPFL, SB MATH SCI-SB-JS, Station 8, CH-1015 Lausanne, Switzerland. E-mail: {\tt joachim.stubbe@epfl.ch}}%

%\thanks{}
%\subjclass{subjects}%
%\keywords{Neumann Laplacian, Dirichlet Laplcian, semiclassical bounds for eigenvalues}

%\date{\today}%
%\dedicatory{}%
%\commby{}%
% ----------------------------------------------------------------
\begin{abstract}
We provide a direct proof of Weyl's law for the buckling eigenvalues of the biharmonic operator on domains of $\mathbb R^d$ of finite measure. The proof relies on asymptotically sharp lower and upper bounds that we develop for the Riesz mean $R_2(z)$. Lower bounds are obtained by making use of the so-called ``averaged variational principle''. Upper bounds are obtained in the spirit of Berezin-Li-Yau. Moreover, we state a conjecture for the second term in Weyl's law and prove its correctness in two special cases: balls in $\mathbb R^d$ and bounded intervals in $\mathbb R$.



\end{abstract}
\maketitle
\setcounter{page}{1}
%\tableofcontents

%%%%%%%%%%%%%%%%%%%%%%%%%%%%%%%%%%%%%%%%%%%%%%%%%%%%%%%%%%%%%%%
%%%%%%%%%%%%%%%%%%%%%%%%%%%%%%%%%%%%%%%%%%%%%%%%%%%%%%%%%%%%%%%
%%%%%%%%%%%%%%%%%%%%%   INTRODUCTION
%%%%%%%%%%%%%%%%%%%%%%%%%%%%%%%%%%%%%%%%%%%%%%%%%%%%%%%%%%%%%%%
%%%%%%%%%%%%%%%%%%%%%%%%%%%%%%%%%%%%%%%%%%%%%%%%%%%%%%%%%%%%%%%
% ----------------------------------------------------------------

\section{Introduction and statement of the main results}

Let $\Omega$ be a domain (i.e., an open connected set) in $\mathbb R^d$ of finite measure. We consider the buckling eigenvalue problem, namely
\begin{equation}
\label{buckling}
\left\{\begin{array}{ll}
\Delta^2 v=-\sigma\Delta v, & \text{in\ }\Omega,\\
v=\partial_\nu v=0, & \text{on\ }\partial\Omega.
\end{array}\right.
\end{equation}
Here $\partial\Omega$ denotes the boundary of $\Omega$ and $\partial_{\nu}v$ denotes the outer normal derivative $v$. Problem \eqref{buckling} is understood in the weak sense as follows: find a function $v\in H^2_0(\Omega)$ and a number $\sigma\in\mathbb R$ such that
$$
\int_{\Omega}\Delta v\Delta\phi dx=\sigma\int_{\Omega}\nabla v\cdot\nabla \phi dx\,\ \ \ \forall\phi\in H^2_0(\Omega).
$$
It is standard to prove that problem \eqref{buckling} admits a non-decreasing sequence of positive eigenvalues of finite multiplicity
$$
0<\sigma_1\le\sigma_2\le\dots\le\sigma_j\le\dots\nearrow+\infty,
$$
with associated eigenfunctions denoted by $v_j$. In particular, the eigenvalues are variationally characterized as
\begin{equation}
\label{buckling_minimax}
\sigma_j=\min_{\substack{V\subset H^2_0 \\ {\rm dim\ }V=j}}\max_{v\in V\setminus \{0\}}\frac{\int_{\Omega} (\Delta v)^2dx}{\int_{\Omega} |\nabla v|^2dx}.
\end{equation}

In this article we are interested in the asymptotic behavior of the eigenvalues $\sigma_j$ as $j\rightarrow+\infty$, or, equivalently, to the asymptotic behavior as $z\rightarrow+\infty$ of the eigenvalue counting function $N(z)$, namely
$$
N(z):=\#\left\{j\in\mathbb N:\sigma_j<z\right\}.
$$
Throughout the paper we shall denote by $\mathbb N$ the set of positive integer numbers. The asymptotic behavior of $N(z)$ as $z\rightarrow+\infty$ when $\Omega$ is a bounded Lipschitz domain is described in \cite{kozlov}, while an alternative approach for domains of class $C^2$ is proposed in \cite{liu}. We also mention \cite{ash_krein} where the authors link the buckling problem \eqref{buckling} to a Krein extension of the Laplace operator, providing a number of properties and results that follows from this relation. Moreover, we refer the interested reader to \cite{rozenbljum} for a more general approach for nonsmooth operators. In the present paper we shall present a simplified approach which requires minimal assumptions on $\Omega$, namely, that $\Omega$ has finite measure. We have been recently informed that an altrenative and simple approach has been proposed in \cite{friedlander_21}, making use of a completely different argument.%. In particular, Weyl's law is proved here for a wide class of domains that also includes bounded Lipschitz domains.

Let $a_{+}$ denote the positive part of a real number $a$, and for $p,z>0$ let
\begin{equation*}
  R_p(z):=\sum_{j}(z-\sigma_j)_{+}^p=p\int_{0}^{+\infty}(z-t)_{+}^{p-1}N(t)\,dt
\end{equation*}
denote the Riesz mean of order $p$. Here we write $\sum_j$ for $\sum_{j=1}^{\infty}$. We state now our main result.
\begin{thm}\label{weyl_buckling}
Let $\Omega$ be a domain in $\mathbb R^d$ of finite measure. Then
 \begin{equation}\label{Buckling-ev-Weyl-R-2}
    \lim_{z\rightarrow+\infty}R_2(z)z^{-2-\frac{d}{2}}=\frac{8}{(d+2)(d+4)}\,(2\pi)^{-d}B_d|\Omega|.
 \end{equation}
This limit is equivalent to Weyl's law for the counting function:
 \begin{equation}\label{Buckling-ev-Weyl-R-0}
    \lim_{z\rightarrow+\infty}N(z)z^{-\frac{d}{2}}=(2\pi)^{-d}B_d|\Omega|.
 \end{equation}
\end{thm}
Throughout the paper, we denote by $|\Omega|$ the measure of $\Omega$ and by $B_d$ the measure of the unit ball in $\mathbb R^d$.

%Roughly speaking, condition  requires that the measure of the tube of size $h$ around $\partial\Omega$ is bounded by $c_{\Omega}h^{\gamma}$ for some $\gamma\in(0,1]$, where $c_{\Omega}>0$ is a constant depending only on $\Omega$. This is always the case for bounded Lipschitz domains, domains with cusps, fractal domains with boundary of Minkowski dimension $D\in(d-1,d)$, certain unbounded domains of finite measure, and more, see e.g., \cite[Remark 4.6]{hps}. 

Note that, as expected, the first term of the asymptotic expansion of $N(z)$ is the same as that of the counting function for the Dirichlet Laplacian.

The strategy of the proof of Theorem \ref{weyl_buckling} relies on proving sharp upper and lower bounds for $R_2(z)$. Lower bounds are obtained thanks to a specific application of the so-called ``averaged variational principle'' (see Lemma \ref{AVP}). This technique, introduced in \cite{HaSt14}, gives an efficient derivation of Kr\"oger's inequality for the Neumann eigenvalues of the Laplacian and
has been used to derive various other lower bounds for Riesz means of eigenvalues, see e.g., \cite{bps,hps}. Actually, proving  asymptotically sharp lower bounds on $R_2(z)$ for buckling eigenvalues is quite a hard task, in fact we prove here a weaker lower bound in terms of an inferior limit, see Theorem \ref{R2-lower-thm}. We note that alternative ways to prove pointwise lower bounds may pass through the so-called universal inequalities for eigenvalues, in the tradition of \cite{ppw}. However, it is still an open problem to prove the conjectured sharp version of the Payne-Polya-Weinberger inequality for buckling eigenvalues. People interested in universal inequalities for problem \eqref{buckling} may refer to \cite{CY_1,CY_2} (see also \cite{ashbaugh_buckling,ashbaugh_hermi}). We remark that in order to prove Theorem \ref{R2-lower-thm} we need to know the asymptotic behavior of certain Riesz means for the Dirichlet eigenvalues of the biharmonic operator. Though this result is known and can be found in \cite{agmon_bi,pleijel_bi,safvas}, we shall directly prove it under minimal regularity assumptions on the domain, see Theorem \ref{weyl_dir_bil}.

Upper bounds are obtained first for $R_1(z)$ in the spirit of the Berezin-Li-Yau inequality \cite{Berezin,LiYau}, and then immediately deduced for $R_2(z)$. We note that Levine and Protter in \cite{LePr} state and prove sharp upper bounds for $R_1(z)$ by means of the Berezin-Li-Yau method. However, the claimed upper bound does not follow from their estimates. We present here an improved version of their proof. In particular, the corresponding estimate is asymptotically sharp. To the best of the authors' knowledge, this fallacy does not seem to have ever been mentioned in the literature.

A further purpose of the present article involves the second term in the asymptotic expansion of $N(z)$. Contrary to the cases of the Laplacian and of the Bilaplacian (see e.g., \cite{bps,frank,frank_2,frank_3}), two-term asymptotic expansions are not known for the buckling problem. The eigenvalues \eqref{buckling} are the eigenvalues of a so-called ``operator pencil'' (or ``operator bundle''), and therefore the techniques used in \cite{bps} do not apply to the buckling problem. Here we present a conjecture based on formal considerations (see Section \ref{sec:4}, see also \cite{bps,safvas}).

\begin {conj}\label{conj_weyl} For a bounded domain $\Omega$ in $\mathbb R^d$ with smooth boundary
\begin{equation}\label{buckling-two-term-asympt-N-of-z}
  N(z)=(2\pi)^{-d}B_d|\Omega|z^{\frac{d}{2}}
 -(2\pi)^{1-d}B_{d-1}|\partial\Omega|z^{\frac{d}{2}-\frac{1}{2}}\left(\frac{1}{4}+\frac{\Gamma\left(\frac{d}{2}\right)}{2\pi^{\frac{1}{2}}\Gamma\left(\frac{d}{2}+\frac{1}{2}\right)}\right)+o\left(z^{\frac{d}{2}-\frac{1}{2}}\right).
 \end{equation}
\end{conj}
For a more precise discussion on smoothness requirements for two-terms asymptotic expansions for biharmonic eigenvalues we refer to \cite{bps} and to \cite[Chapter 1.6]{safvas}.

Note that the first contribution of the second term (the factor $\frac{1}{4}$) coincides with the second term in the eigenvalue asymptotics for the Dirichlet Laplacian. For the Riesz mean $\displaystyle R_1(z)$ of the buckling problem, formula \eqref{buckling-two-term-asympt-N-of-z} yields the following two-terms asymptotics as $z\rightarrow+\infty$:
\begin{multline}\label{buckling-two-term-asympt-R-1-of-z}
 R_1(z)=\frac{2}{d+2}\,(2\pi)^{-d}B_d|\Omega|z^{\frac{d}{2}+1}
 -\frac{2}{d+1}(2\pi)^{1-d}B_{d-1}|\partial\Omega|z^{\frac{d}{2}+\frac{1}{2}}\left(\frac{1}{4}+\frac{\Gamma\left(\frac{d}{2}\right)}{2\pi^{\frac{1}{2}}\Gamma\left(\frac{d}{2}+\frac{1}{2}\right)}\right)\\+o\left(z^{\frac{d}{2}+\frac{1}{2}}\right).
\end{multline}
We shall prove the validity of the conjecture \eqref{conj_weyl} when $\Omega$ is a ball in $\mathbb R^d$, $d\geq 2$, namely we prove the following theorem.

\begin{thm}\label{conj_ball}
Let $\Omega$ be a ball in $\mathbb R^d$. Then \eqref{buckling-two-term-asympt-N-of-z} holds for the eigenvalues of \eqref{buckling} on $\Omega$.
\end{thm}

Note that the two-term formula \eqref{buckling-two-term-asympt-N-of-z} for balls was already computed in \cite{ash_krein} by means of the relation of the buckling problem \eqref{buckling} with a Krein extension of the Laplace operator. We propose here an alternative proof that, although based on similar ideas, relies on explicit identities relating the eigenvalues of the buckling problem to the eigenvalues of the Dirichlet Laplacian which are both given in terms of zeroes of Bessel functions.

%Note that the two-term formula \eqref{buckling-two-term-asympt-N-of-z} for balls has been computed in \cite{ash_krein}. Namely, the authors observe that the positive eigenvalues of the Krein Laplacian coincide with the buckling eigenvalues. Though the arguments are similar, we shall provide here an alternative, self-contained proof.

Another situation where the eigenvalues are (almost) explicitly given is $d=1$.  We remark that typically in the case $d=1$ the counting function $N(z)$ is not regular enough for a two-terms expansion, therefore we expect the conjecture to hold in its weaker form, namely in the form of \eqref{buckling-two-term-asympt-R-1-of-z}. This situation is known for the eigenvalues of the Dirichlet Laplacian $\pi^2j^2$ on $(0,1)$. In fact, in this case $N(z)=[\frac{z^{\frac{1}{2}}}{\pi}]$ which does not have a two-term expansion with a lower order power. On the other hand, $R_1(z)=\frac{2}{3\pi}z^{\frac{3}{2}}-\frac{z}{2}+o(z)$. An analogous expansion holds for the eigenvalues of the Bilaplacian, see \cite{bps}. We shall prove here a two-terms expansion for the first Riesz mean of problem \eqref{buckling} on a bounded interval, namely we prove the following theorem.
\begin{thm}\label{conj_1d}
Let $d=1$, $L>0$ and $\Omega=(0,L)$. Then \eqref{buckling-two-term-asympt-R-1-of-z} holds for the eigenvalues of \eqref{buckling} on $\Omega$.
\end{thm}
In the one-dimensional case we are also able to prove the following inequality, which is not strictly related with two-terms asymptotics, but which has an interest per se:
$$
\Lambda_j>\lambda_j\sigma_j
$$
for all $j\in\mathbb N$ (see Proposition \ref{generalized_payne_thm}). Here $\lambda_j$ and $\Lambda_j$ denote the eigenvalues of the Laplacian and the Bilaplacian with Dirichlet conditions, respectively. This is a generalization of an inequality by Payne \cite{payne}, holding for $j=1$ in $d$ dimensions. It is natural to expect that a similar behavior occurs for higher dimensions and all $j$. We are led to state the following conjecture.
\begin{conj}
%\label{blb_conj}
For a bounded domain $\Omega$ in $\mathbb R^d$ with smooth boundary
\begin{equation*}
%\label{generalized_payne_conj}
\Lambda_j>\lambda_j\sigma_j,
\end{equation*}
for all $j\in\mathbb N$.
\end{conj}

The present paper is organized as follows. In Section \ref{sec:2} we prove lower and upper bounds for Riesz means of buckling eigenvalues (Theorems \ref{R2-lower-thm} and \ref{R1-upper-thm}, respectively) and combine them to prove our main Theorem \ref{weyl_buckling}. In Section \ref{sec:4} we prove Theorem \ref{conj_ball}. Section \ref{sec:3} is devoted to the study of the one-dimensional buckling problem and to the proof of Theorem \ref{conj_1d}.

%%%%%%%%%%%%%%%%%%%%%%%%%%%%%%%%%%%%%%%%%%%%%%%%%%%%%%%%%%%%%%%
%%%%%%%%%%%%%%%%%%%%%%%%%%%%%%%%%%%%%%%%%%%%%%%%%%%%%%%%%%%%%%%
%%%%%%%%%%%%%%%%%%%%%   Weyl's law
%%%%%%%%%%%%%%%%%%%%%%%%%%%%%%%%%%%%%%%%%%%%%%%%%%%%%%%%%%%%%%%
%%%%%%%%%%%%%%%%%%%%%%%%%%%%%%%%%%%%%%%%%%%%%%%%%%%%%%%%%%%%%%%
% ----------------------------------------------------------------

\section{Weyl's law for buckling eigenvalues} \label{sec:2}

The aim of this section is to prove Theorem \ref{weyl_buckling}. To do so we shall prove lower and upper bounds for $R_2(z)$. First we need to recall a few useful results on the eigenvalues of the Dirichlet Laplacian and Bilaplacian.

Through the rest of the paper we shall denote by $\|f\|_2$ the standard $L^2$-norm of a square integrable function $f$ on $\Omega$. By $H^1_0(\Omega)$ we denote the closure of $C^{\infty}_c(\Omega)$ (the space of smooth functions compactly supported in $\Omega$) with respect to the norm $\left(\|f\|_2^2+\|\nabla f\|_2^2\right)^{\frac{1}{2}}$. By $H^2_0(\Omega)$ we denote the closure of $C^{\infty}_c(\Omega)$ with respect to the norm $\left(\|f\|_2^2+\|\Delta f\|_2^2\right)^{\frac{1}{2}}$.

%%%%%%%%%%%%%%%%%%%%%%%%%%%%%%%%%%%%%%%%%%%%%%%%%%%%%%%%%%%%%%%
%%%%%%%%%%%%%%%%%%%%%  OTHER PROBLEMS
%%%%%%%%%%%%%%%%%%%%%%%%%%%%%%%%%%%%%%%%%%%%%%%%%%%%%%%%%%%%%%%

\subsection{Dirichlet problems}
We recall here a few facts on the eigenvalues of the Laplacian and of the Bilaplacian with Dirichlet boundary conditions. Let $\Omega$ be a domain in $\mathbb R^d$ of finite measure. The eigenvalue problem for the Dirichlet  Laplacian reads
\begin{equation}
\label{dirichlet_laplacian}
\left\{\begin{array}{ll}
-\Delta u=\lambda u, & \text{in\ }\Omega,\\
u=0, & \text{on\ }\partial\Omega,
\end{array}\right.
\end{equation}
while the eigenvalue problem for the Dirichlet Bilaplacian reads
\begin{equation}
\label{dirichlet_bilaplacian}
\left\{\begin{array}{ll}
\Delta^2 U=\Lambda U, & \text{in\ }\Omega,\\
U=\partial_\nu U=0, & \text{on\ }\partial\Omega.
\end{array}\right.
\end{equation}
Also problems \eqref{dirichlet_laplacian} and \eqref{dirichlet_bilaplacian} are understood in the weak sense, and admit non-decreasing sequences of positive eigenvalues of finite multiplicity, given by
$$
0<\lambda_1<\lambda_2\leq\cdots\leq\lambda_j\leq\cdots\nearrow+\infty
$$
and
$$
0<\Lambda_1\leq\Lambda_2\leq\cdots\leq\Lambda_j\leq\cdots\nearrow+\infty
$$
respectively. The corresponding eigenfunctions, denoted by $\left\{u_j\right\}_j$ and $\left\{U_j\right\}_j$, belong respectively to $H^1_0(\Omega)$ and to $H^2_0(\Omega)$, and form orthonormal bases of $L^2(\Omega)$.

The eigenvalues are variationally characterized as
\begin{equation}\label{minmax-dirichlet}
\lambda_j=\min_{\substack{U\subset H^1_0(\Omega)\\{\rm dim\  U=j}}}\max_{u\in U\setminus\{0\}}\frac{\int_{\Omega}|\nabla u|^2dx}{\int_{\Omega}u^2dx}
\end{equation}
and
$$
\Lambda_j=\min_{\substack{U\subset H^2_0(\Omega)\\{\rm dim \ U=j}}}\max_{u\in U\setminus\{0\}}\frac{\int_{\Omega}(\Delta u)^2dx}{\int_{\Omega}u^2dx}.
$$
We recall the well-known Berezin-Li-Yau upper bound for $\lambda_j$ (see \cite{Berezin,LiYau}).
\begin{thm}\label{berezin}
Let $\Omega$ be a domain in $\mathbb R^d$ of finite measure. For all $z>0$
$$
\sum_j(z-\lambda_j)_+\leq \frac{2}{d+2}(2\pi)^{-d}B_d|\Omega|z^{1+\frac{d}{2}}.
$$
The inequality holds true if we replace $\lambda_j$ with $\|\nabla f_j\|_2^2$, where $\{f_j\}_j$ is a $L^2$-orthonormal family in $H^1_0(\Omega)$. 
\end{thm}

%%%%%%%%%%%%%%%%%%%%%% OLD CONDITION M %%%%%%%%%%%%%%%%%%%%%%%%%%%%%%

\begin{comment}
Let now $\Omega$ be a domain in $\mathbb R^d$ with boundary $\partial\Omega$. By $\omega_h$ we denote the inner tube of size $h$ around $\partial\Omega$, namely
$$
\omega_h=\{x\in\Omega:{\rm dist}(x,\partial\Omega)<h\}.
$$

Assume that $\Omega$ satisfies 
\begin{equation}\label{M}\tag{M}
|\Omega|<+\infty\,,\ \ \ \limsup_{h\rightarrow 0^+}\frac{|\omega_h|}{h^{\gamma}}<+\infty
\end{equation}
for some $\gamma\in(0,1]$.  Note that bounded Lipschitz domains satisfy condition \eqref{M} with $\gamma=1$. However much less regular domains satisfy condition \eqref{M} with $\gamma>0$, such as domains with cusps or fractal domains (see e.g., \cite[Remark 4.6]{hps}). The class of domains enjoying \eqref{M} includes also unbounded domains. As a simple example, we may think of the domain $\Omega\subset\mathbb R^3$ defined by $\Omega:=\{(x_1,x_2,x_3)\in\mathbb R^3:x_1> 1, x_2^2+x_3^2< x_1^{-2\alpha}\}$ for $\alpha>1$.

Under condition \eqref{M} we state the following Weyl's law for the eigenvalues $\Lambda_j$.

\end{comment}

%%%%%%%%%%%%%%%%%%%%%% OLD CONDITION M %%%%%%%%%%%%%%%%%%%%%%%%%%%%%%

Let us state now Weyl's law for the eigenvalue $\Lambda_j$. To be self-contained, we shall also include a proof.

\begin{thm}\label{weyl_dir_bil}
Let $\Omega$ be a domain in $\mathbb R^d$ of finite measure. Then
$$
\lim_{z\rightarrow+\infty}z^{-2-\frac{d}{2}}\sum_j(z^2-\Lambda_j)_+= \frac{4}{d+4}(2\pi)^{-d}B_d|\Omega|.
$$
\end{thm}
\begin{proof}
The proof is a direct consequence of asymptotically sharp upper and lower bounds on $\sum_j(z-\Lambda_j)_+$. Upper bounds in the spirit of Berezin-Li-Yau are classical and are proved for domains of finite measure in \cite{Lap1997}:
\begin{equation}\label{BLY-bilaplacian}
\sum_{j}(z-\Lambda_j)_+\leq\frac{4}{d+4}(2\pi)^{-d}B_d|\Omega|z^{\frac{d}{4}+1},
\end{equation}
for all $z>0$.

Asymptotically sharp lower bounds are proved in \cite{bps} under certain conditions on the size of a tubular neighborhood of the boundary. Actually, in \cite{bps} we were interested in estimating the size of the second term of the bounds, however if one is interested only in the asymptotic behavior (i.e., on the first term), it is sufficient to assume $\Omega$ of finite measure.

For the reader's convenience, we shall briefly explain this last fact. It follows from \cite[Theorem C and Lemma 4.4]{bps} that for any domain $\Omega$ of $\mathbb R^d$ of finite measure, all $z>0$ and all sufficiently small $h>0$
\begin{multline}\label{AVP_bilaplacian}
\sum_{j}(z-\Lambda_j)_+\geq\frac{4}{d+4}(2\pi)^{-d}B_d|\Omega|z^{\frac{d}{4}+1}\\
+\frac{4}{d+4}(2\pi)^{-d}B_d|\Omega|\left(\left(z-\frac{|\omega_h|\tilde A_d^2}{(|\Omega|-|\omega_h|)h^4}\right)^{\frac{d}{4}+1}-z^{\frac{d}{4}+1}\right)\\
-\frac{4}{d+4}(2\pi)^{-d}B_d|\omega_h|\left(z-\frac{|\omega_h|\tilde A_d^2}{(|\Omega|-|\omega_h|)h^4}\right)^{\frac{d}{4}+1}\\
-2(2\pi)^{-d}\frac{B_d|\omega_h|A_d^2}{h^2}\left(z-\frac{|\omega_h|\tilde A_d^2}{(|\Omega|-|\omega_h|)h^4}\right)^{\frac{d}{4}+\frac{1}{2}},
\end{multline}
for some $\tilde A_d,A_d>0$ depending only on $d$. Here by $\omega_h$ we denote the inner open tube of size $h$ around $\partial\Omega$, namely
$$
\omega_h=\{x\in\Omega:{\rm dist}(x,\partial\Omega)<h\}.
$$
Choosing $h=C z^{-\frac{1}{4}}$ for some $C>0$ (which can be chosen in an optimal way), we deduce that the lower bound is asymptotically sharp provided that $\lim_{h\rightarrow 0^+}|\omega_h|=0$. This is always the case if $\Omega$ has finite measure, by dominated convergence theorem, regardless of the Lebesgue measure of $\partial\Omega$ (which in principle may be different from zero).

We refer to \cite{bps} for more details on how \eqref{AVP_bilaplacian} is obtained as an application of the averaged variational principle (see also Theorem \ref{buckling_AVP_1}) and to \cite{hps} for the analogue analysis in the case of the Dirichlet Laplacian.
\end{proof}

The result of Theorem \ref{AVP_bilaplacian} is not new. In fact, Weyl's law for Dirichlet problems (Laplacian and polyharmonc operators) on unbounded domains has been proved in \cite{roz_72} through a quite sophisticated version of Dirichlet-Neumann bracketing.

\begin{rem}
Sometimes it is interesting to estimate the second term in \eqref{AVP_bilaplacian} in a more explicit way (see \cite{bps,hps}). This requires more knowledge on the behavior of $|\omega_h|$ as $h\rightarrow 0^+$. A natural assumption is that $\Omega$ satisfies
\begin{equation}\label{M}%\tag{M}
 \limsup_{h\rightarrow 0^+}\frac{|\omega_h|}{h^{\gamma}}<+\infty
\end{equation}
for some $\gamma\in(0,1]$.  Note that bounded Lipschitz domains satisfy condition \eqref{M} with $\gamma=1$. However much less regular domains satisfy condition \eqref{M} with $\gamma>0$, such as domains with cusps or fractal domains (see e.g., \cite[Remark 4.6]{hps}). The class of domains enjoying \eqref{M} includes also suitable unbounded domains.
\end{rem}

%%%%%%%%%%%%%%%%%%%%%%%%%%%%%%%%%%%%%%%%%%%%%%%%%%%%%%%%%%%%%%%
%%%%%%%%%%%%%%%%%%%%%   RIESZ MEAN LOWER BOUNDS
%%%%%%%%%%%%%%%%%%%%%%%%%%%%%%%%%%%%%%%%%%%%%%%%%%%%%%%%%%%%%%%

\subsection{Lower bounds for Riesz means}

We start by quoting the ``averaged variational principle'', which will be used in this subsection, with the formulation provided in \cite{EHIS}.
\begin{lem}\label{AVP}
Consider a self-adjoint operator $H$ on a Hilbert space $\mathcal{H}$,
the spectrum of which is discrete at least in its lower portion, so that
$- \infty < \omega_1 \le \omega_2 \le \dots$.
The corresponding orthonormalized eigenvectors
are denoted $\{\mathbf{\psi}^{(j)}\}$.  The
closed quadratic form corresponding to $H$
is denoted $Q(\varphi, \varphi)$ for vectors $\varphi$ in
the quadratic-form domain $\mathcal{Q}(H) \subset \mathcal{H}$.
Let $f_\zeta \in \mathcal{Q}(H)$ be
a family of
vectors indexed by
a variable $\zeta$ ranging over
a measure space $(\mathfrak{M},\Sigma,\sigma)$.
Suppose that $\mathfrak{M}_0$ is a
subset of $\mathfrak{M}$.  Then for any  $z \in \mathbb{R}$,
\begin{equation}\label{RieszVersion}
	\sum_{j}{\left(z - \omega_j\right)_{+} \int_{\mathfrak{M}}\left|\langle\mathbf{\psi}^{(j)}, f_\zeta\rangle\right|^2\,d \sigma}
       \geq
	\int_{\mathfrak{M}_0}{\left(z\| f_\zeta\|^2 - Q(f_\zeta,f_\zeta) \right) d \sigma},
\end{equation}
provided that the integrals converge.
\end{lem}

We recall that the eigenfunctions $v_j$ of problem \eqref{buckling} satisfy the following normalization condition
\begin{equation*}%\label{buckling-eigenfct-normalization}
  \int_{\Omega}\nabla v_j\cdot\nabla v_k \,dx=\delta_{jk},
\end{equation*}
where $\delta_{jk}$ denotes the Kronecker delta. In particular,
\begin{equation*}%\label{buckling-eigenfct-completeness-for-gradients}
  \sum_{j=1}^{\infty}\bigg|\int_{\Omega}\nabla v_j\cdot\nabla u \,dx\bigg|^2=\int_{\Omega}|\nabla u|^2 \,dx,
\end{equation*}
for each $u\in H^2_0(\Omega)$. Moreover, $\{\nabla v_j\}_j$ forms an orthonormal system in $(L^2(\Omega))^d$  and in particular
\begin{equation}\label{buckling-eigenfct-Bessel-ineq-vectors}
  \sum_{j=1}^{\infty}\bigg|\int_{\Omega}\nabla v_j\cdot V \,dx\bigg|^2\leq \int_{\Omega}|V|^2 \,dx,
\end{equation}
for $V\in (L^2(\Omega))^d$.

We prove a first useful lower bound for $R_2(z)$.

\begin{thm}\label{R2-lower-thm}
Let $\Omega$ be a domain in $\mathbb R^d$ of finite measure. Then
\begin{equation}\label{R2-lower}
     \liminf_{z\rightarrow+\infty}R_2(z)z^{-2-\frac{d}{2}}\geq\frac{8}{(d+2)(d+4)}\,(2\pi)^{-d}B_d|\Omega|.
  \end{equation}
  \end{thm}
\begin{proof}
 We apply Lemma \ref{AVP} for the buckling problem with test functions $U_j$ being the eigenfunctions of the biharmonic Dirichlet operator with eigenvalues $\Lambda_j$. Namely, we take into \eqref{RieszVersion} $\|f\|^2=\|\nabla f\|_2^2$, $Q(f,f)=\|\Delta f\|_2^2$, $\omega_j=\sigma_j$, $\psi^{(j)}=v_j$, $\mathfrak M=\mathbb N$, $\mathfrak M_0=I\subseteq \mathbb N$ and $f_{\zeta}=U_j$, $j\in\mathbb N$.
 
Let $z>0$. Then we have
\begin{equation*}%\label{AVP_buckling_1}
\sum_j(z-\sigma_j)_+\sum_{i=1}^{\infty}\left|\int_{\Omega}\nabla v_j\cdot\nabla U_i dx\right|^2\geq\sum_{i\in I}\left(z\|\nabla U_i\|_2^2-\Lambda_i\right).
\end{equation*}
Now, integrating by parts yields
\begin{equation*}%\label{AVP_buckling_2}
\sum_{i=1}^{\infty}\left|\int_{\Omega}\nabla v_j\cdot\nabla U_i dx\right|^2=\sum_{i=1}^{\infty}\left|\int_{\Omega}\Delta v_j U_i dx\right|^2=\|\Delta v_j\|_2^2=\sigma_j,
\end{equation*}
since $\left\{U_i\right\}_i$ is an orthonormal basis of $L^2(\Omega)$.  We have proved that
\begin{equation}\label{Buckling-avp-3}
 \sum_{j}(z-\sigma_j)_{+}\sigma_j\geq \sum_{i\in I}z\int_{\Omega}|\nabla U_i|^2 \,dx-\Lambda_i,
\end{equation}
for some index set $I$ to be specified later. We rewrite \eqref{Buckling-avp-3}  as
\begin{equation}\label{Buckling-avp-4}
 z\,R_1(z)-R_2(z)\geq \sum_{i\in I}(z^2-\Lambda_i)-z(z-\int_{\Omega}|\nabla U_i|^2 \,dx).
\end{equation}
% Since $U_i\in H^2_0(\Omega)$ we have by the Cauchy-Schwarz inequality $\displaystyle z-\int_{\Omega}|\nabla U_i|^2 \,dx >z-\Lambda_i^{\frac{1}{2}}>0$. Therefore
Moreover,
\begin{equation}\label{Buckling-avp-4-part1}
  \sum_{i\in I}z(z-\int_{\Omega}|\nabla U_i|^2 \,dx)\leq \sum_{i}z(z-\int_{\Omega}|\nabla U_i|^2 \,dx)_{+}.
\end{equation}
On the other hand, by the Berezin-Li-Yau inequality for $L^2$-orthonormal functions in $H^1_0(\Omega)$ (Theorem \ref{berezin}) we have
\begin{equation}\label{Buckling-avp-4-part2}
 \sum_{i}(z-\int_{\Omega}|\nabla U_i|^2 \,dx)_{+}\leq \frac{2}{d+2}\,(2\pi)^{-d}B_d|\Omega|z^{1+\frac{d}{2}}.
\end{equation}
We choose $I=\{i: \Lambda_i\leq z^2\}$. Since $\displaystyle \frac{d}{dz}\,R_2(z)=2R_1(z)$, from \eqref{Buckling-avp-4}, \eqref{Buckling-avp-4-part1}, and \eqref{Buckling-avp-4-part2} we get the following differential inequality for $R_2(z)$, $z>0$:
\begin{equation}\label{Buckling-avp-diff-ineq-1}
 z^3\frac{d}{dz}\bigg(\frac{R_2(z)}{2z^2}\bigg)\geq \sum_{i}(z^2-\Lambda_i)_{+}-\frac{2}{d+2}\,(2\pi)^{-d}B_d|\Omega|z^{2+\frac{d}{2}}.
\end{equation}
Weyl's law for Dirichlet Bilaplacian eigenvalues (Theorem \ref{weyl_dir_bil}) states that
\begin{equation*}%\label{Biharmonic-Dirichlet-ev-Weyl-R-1}
    \lim_{z\rightarrow+\infty}z^{-2-\frac{d}{2}}\sum_{i}(z^2-\Lambda_i)_{+}=\frac{4}{d+4}\,(2\pi)^{-d}B_d|\Omega|.
 \end{equation*}
As a consequence, for any $\epsilon >0$ there is $z_{\epsilon}>0$ such that for all $z\geq z_{\epsilon}$
\begin{equation*}
 z^{-2-\frac{d}{2}}\sum_{i}(z^2-\Lambda_i)_{+}\geq \frac{4}{d+4}\,(2\pi)^{-d}B_d|\Omega|-\epsilon.
\end{equation*}
Inserting this lower bound into \eqref{Buckling-avp-diff-ineq-1} we obtain
\begin{equation*}%\label{Buckling-avp-diff-ineq-2}
 z^3\frac{d}{dz}\bigg(\frac{R_2(z)}{2z^2}\bigg)\geq\frac{2d}{(d+2)(d+4)}\,(2\pi)^{-d}B_d|\Omega|z^{2+\frac{d}{2}}-\epsilon z^{2+\frac{d}{2}},
\end{equation*}
for all $z\geq z_{\epsilon}$. Integrating \eqref{Buckling-avp-diff-ineq-1} in $[z_{\epsilon},z]$ yields the inequality
\begin{equation*}%\label{Buckling-avp-R-2-ineq-1}
\frac{R_2(z)}{2z^2}-\frac{R_2(z_{\epsilon})}{2z_{\epsilon}^2}\geq \bigg(\frac{4}{(d+2)(d+4)}\,(2\pi)^{-d}B_d|\Omega|-\frac{2\epsilon}{d}\bigg)\big(z^{\frac{d}{2}}-z_{\epsilon}^{\frac{d}{2}}\big),
\end{equation*}
from which we deduce \eqref{R2-lower}.
\end{proof}

Following the same lines of the proof of Theorem \ref{R2-lower-thm} we can deduce a number of other bounds for functions related to $R_2(z)$.

\begin{cor}\label{buckling_dirichlet}
Let $\Omega$ be a domain in $\mathbb R^d$ of finite measure. Then for any $z>0$ and $k\in\mathbb N$
\begin{equation}\label{RieszD}
\sum_j(z-\sigma_j)_+\sigma_j\geq\sum_{j=1}^k(z\lambda_j-\Lambda_j).
\end{equation}
Moreover, for any $k\in\mathbb N$
\begin{equation}\label{AVP_D}
\frac{1}{k}\sum_{j=1}^k(\sigma_j^2-\Lambda_j)\leq\sigma_{k+1}\left(\frac{1}{k}\sum_{j=1}^k\sigma_j-\lambda_j\right).
\end{equation}
\end{cor}

\begin{proof}
We note that
\begin{equation}\label{AVP_buckling_3}
\sum_{j=1}^k\lambda_j\leq\sum_{j=1}^k\int_{\Omega}|\nabla U_j|^2dx,
\end{equation}
where $\lambda_j$ are the eigenvalues of the Dirichlet Laplacian and $U_j$ the eigenfunctions of the Dirichlet Bilaplacian. Inequality \eqref{AVP_buckling_3} follows by applying \eqref{RieszVersion} to the Dirichlet Laplacian, that is, taking into \eqref{RieszVersion} $\|f\|^2=\| f\|_2^2$, $Q(f,f)=\|\nabla f\|_2^2$, $\omega_j=\lambda_j$, $\psi^{(j)}=u_j$, $\mathfrak M=\mathbb N$, $\mathfrak M_0=I\subseteq \mathbb N$ and $f_{\zeta}=U_j$, $j\in\mathbb N$. Doing so we obtain
\begin{equation*}%\label{AVP_laplacian_1}
\sum_j(z-\lambda_j)_+\sum_{i=1}^{\infty}\left|\int_{\Omega}u_j U_i dx\right|^2\geq \sum_{i\in I}\left(z-\|\nabla U_i\|_2^2\right)
\end{equation*}
which is equivalent to
\begin{equation*}%\label{AVP_laplacian_2}
\sum_j(z-\lambda_j)_+\geq\sum_{i\in I}\left(z-\|\nabla U_i\|_2^2\right).
\end{equation*}
Choosing $I=\left\{1,...,k\right\}$ we obtain \eqref{AVP_buckling_3}. As a consequence, plugging  \eqref{AVP_buckling_3} in \eqref{Buckling-avp-3} we obtain \eqref{RieszD}. Now, taking $z=\sigma_{k+1}$ in \eqref{RieszD} with $I=\left\{1,...,k\right\}$ yields \eqref{AVP_D}.
\end{proof}

\begin{rem}

We note that \eqref{RieszD} implies for all $z>0$
$$
\frac{d}{dz}\sum_j\frac{(z-\sigma_j)_+^2}{z^2}=\frac{2}{z^3}\sum_j(z-\sigma_j)_+\sigma_j\geq\frac{2}{z^3}\sum_j(z\lambda_j-\Lambda_j)_+,
$$
and integrating this inequality between $0$ and $z$ we obtain
$$
\sum_j(z-\sigma_j)_+^2\geq\sum_j\frac{\lambda_j^2}{\Lambda_j}\left(z-\frac{\Lambda_j}{\lambda_j}\right)_+^2.
$$
\end{rem}

Another bound can be obtained by using the ``averaged variational principle'' \eqref{AVP} with test functions of the form $(2\pi)^{-\frac{d}{2}}e^{i p \cdot x}\phi(x)$ with $\phi\in H^2_0(\Omega)\cap L^{\infty}(\Omega)$ real valued. 
\begin{thm}\label{buckling_AVP_1}
Let $\Omega$ be a domain in $\mathbb R^d$ of finite measure. For any $\phi\in H^2_0(\Omega)\cap L^{\infty}(\Omega)$ and $z>0$,
\begin{multline*}%\label{AVP_h}
\|\phi\|_{\infty}^2\sum_j(z-\sigma_j)_+\sigma_j\\
\geq\frac{2d B_d\|\phi\|_2^2}{(d+2)(d+4)}(2\pi)^{-d}z^{2+\frac{d}{2}}-B_d(2\pi)^{-d}\|\nabla\phi\|_2^2z^{1+\frac{d}{2}}-B_d(2\pi)^{-d}\|\Delta\phi\|_2^2z^{\frac{d}{2}}.
\end{multline*}
\end{thm}
\begin{proof}
The proof can be carried out in the very same way as that of Corollary \ref{buckling_dirichlet}. In this case we  take into \eqref{RieszVersion} $\|f\|^2=\|\nabla f\|_2^2$, $Q(f,f)=\|\Delta f\|_2^2$, $\omega_j=\sigma_j$, $\psi^{(j)}=v_j$, $\mathfrak M=\mathbb R^d$, $\mathfrak M_0=B(0,z^{\frac{1}{2}})$ and $f_{\zeta}=(2\pi)^{-\frac{d}{2}}e^{i\zeta\cdot x}\phi$, $\zeta\in\mathbb R^d$. Here $B(0,z^{\frac{1}{2}})$ denotes the ball of radius $z^{\frac{1}{2}}$ in $\mathbb R^d$. For more details we refer to  \cite[Theorem 2.1]{hps} in the case of the Dirichlet Laplacian and to \cite[Theorem 4.1]{bps} in the case of the Dirichlet Bilaplacian.
\end{proof}
\begin{rem}
We note that in Theorem \ref{buckling_AVP_1} we have a certain freedom for the choice of $\phi\in H^2_0(\Omega)$, at least if $\Omega$ is bounded and smooth enough. In fact, for a given $h>0$ sufficiently small  it is always possible to find a non-negative function $\phi\in H^2_0(\Omega)$ with $\|\phi\|_{\infty}\leq 1$, $|\nabla\phi|\leq c_d h^{-1}$, $|\Delta\phi|\leq c_dh^{-2}$, where $c_d$ depends only on $d$, and with $|\Omega|-\|\phi\|_2^2\sim |\partial\Omega|h$ as $h\rightarrow 0^+$. In particular, such $\phi$ equals $1$ for all points of $\Omega$ at distance from $\partial\Omega$ bigger than $h$. We refer to \cite[Lemma 4.4]{bps} for more details. In particular, we see that choosing $h=z^{-\frac{1}{2}}$ leads to asymptotically sharp lower bounds with lower order terms of the correct order.
\end{rem}

%%%%%%%%%%%%%%%%%%%%%%%%%%%%%%%%%%%%%%%%%%%%%%%%%%%%%%%%%%%%%%%
%%%%%%%%%%%%%%%%%%%%%   BEREZIN LI YAU
%%%%%%%%%%%%%%%%%%%%%%%%%%%%%%%%%%%%%%%%%%%%%%%%%%%%%%%%%%%%%%%

\subsection{The Berezin-Li-Yau method for the buckling problem}

Here we reconsider, in a simplified presentation, the Berezin-Li-Yau method employed in \cite{LePr}, showing a computational inaccuracy present in the proof. In particular, we show that the given sharp eigenvalue bound does not follow from the estimates presented in \cite{LePr}. Then we provide a proof of the sharp Berezin-Li-Yau estimate for the average of the first $k$ buckling eigenvalues. We then observe that an upper bound on the average is  equivalent to a suitable lower bound on $R_1(z)$.

We start by observing that, for any $u\in H^2_0(\Omega)$,
\begin{equation*}
  \int_{\Omega}|\nabla u(x)|^2\,dx=\int_{\mathbb{R}^d}|p|^2|\hat{u}(p)|^2\,dp,\quad \int_{\Omega}|\Delta u(x)|^2\,dx=\int_{\mathbb{R}^d}|p|^4|\hat{u}(p)|^2\,dp,
\end{equation*}
where $\hat u(p)=(2\pi)^{-\frac{d}{2}}\int_{\mathbb R^d}u(x)e^{ip\cdot x}dx$ denotes the Fourier transform of the function $u$ extended by $0$ to $\mathbb R^d$. In particular, for the buckling eigenfunctions $v_j$
\begin{equation*}%\label{Ft-properties-buckling-eigenfcts}
  \int_{\mathbb{R}^d}|p|^2|\hat{v}_j(p)|^2\,dp=1,\quad \int_{\mathbb{R}^d}|p|^4|\hat{v}_j(p)|^2\,dp=\sigma_j.
\end{equation*}
Therefore, for any $R>0$ and $k\in\mathbb N$,
\begin{equation}\label{Buckling-problem-Li-Yau-basic}
 \begin{split}
    \sum_{j=1}^{k}\sigma_j & =  kR^2+ \int_{\mathbb{R}^d}(|p|^4-R^2|p|^2)\sum_{j=1}^{k}|\hat{v}_j(p)|^2\,dp\\
      & \geq kR^2+ \int_{B(0,R)}(|p|^4-R^2|p|^2)\sum_{j=1}^{k}|\hat{v}_j(p)|^2\,dp.\\
 \end{split}
\end{equation}
In order to produce a lower bound on the sum of the $\sigma_j$'s, we need an upper bound on $\sum_{j=1}^{k}|\hat{v}_j(p)|^2$. For the reader's convenience, we show first the bound derived in \cite{LePr}. Let $f_k(p)=\sum_{j=1}^{k}|\hat{v}_j(p)|^2$, and $e_{\alpha}$, $\alpha=1,\ldots, d$ be the canonical basis vectors in $\mathbb{R}^d$. Any $p\in\mathbb R^d$ has coordinates $(p_1,...,p_d)$. Then
\begin{equation}\label{Levine-Protter-FT-upper-bound}
\begin{split}
  |p|^2 f_k(p)&=\sum_{j=1}^{k}\sum_{\alpha=1}^{d}|p_{\alpha}\hat{v}_j(p)|^2\\
  &=(2\pi)^{-d}\sum_{j=1}^{k}\sum_{\alpha=1}^{d}\bigg|\int_{\Omega} v_j(x)(-ie_{\alpha}\cdot\nabla e^{ipx}) \,dx\bigg|^2\\
  &=(2\pi)^{-d}\sum_{\alpha=1}^{d}\sum_{j=1}^{k}\bigg|\int_{\Omega} \nabla v_j(x)\cdot e_{\alpha} e^{ipx} \,dx\bigg|^2\\
  &\leq d(2\pi)^{-d}|\Omega|,
\end{split}
\end{equation}
where the final inequality follows from \eqref{buckling-eigenfct-Bessel-ineq-vectors} applied for each $\alpha$. This corresponds to \cite[Inequality (4.6)]{LePr}. With this estimate however one does not get the sharp lower bound (by optimizing \eqref{Buckling-problem-Li-Yau-basic} with respect to $R$ and using \eqref{Levine-Protter-FT-upper-bound}), which should read
\begin{equation}\label{Buckling-ev-sum-sharp-lower-bound}
  \frac{d+2}{d}\frac{1}{k}\sum_{j=1}^{k}\sigma_j\geq C_d \left(\frac{k}{|\Omega|}\right)^{\frac 2 d},
\end{equation}
with $C_d=\frac{4\pi^2}{B_d^{\frac{2}{d}}}$, as claimed by \cite[inequality (4.8)]{LePr}, but only
\begin{equation*}%\label{Buckling-ev-sum-notsharp-LevineProtter-lower-bound}
  \frac{d+2}{d}\frac{1}{k}\sum_{j=1}^{k}\sigma_j\geq d^{-\frac2 d} C_d \left(\frac{k}{|\Omega|}\right)^{\frac2 d}.
\end{equation*}
Nevertheless, the upper bound \eqref{Levine-Protter-FT-upper-bound} is easily improved in the following way
\begin{equation}\label{sharp-FT-upper-bound}
\begin{split}
  |p|^4 f_k(p)&=\sum_{j=1}^{k}||p|^2\hat{v}_j(p)|^2\\
  &=(2\pi)^{-d}\sum_{j=1}^{k}\bigg|\int_{\Omega} v_j(x)|p|^2 e^{ipx} \,dx\bigg|^2\\
  &=(2\pi)^{-d}\sum_{j=1}^{k}\bigg|\int_{\Omega} v_j(x)(-\Delta e^{ipx})\,dx\bigg|^2\\
  &=(2\pi)^{-d}\sum_{j=1}^{k}\bigg|\int_{\Omega} \nabla v_j(x) \cdot\nabla e^{ipx}\,dx\bigg|^2\\
  &\leq(2\pi)^{-d}\int_{\Omega}|\nabla e^{ipx}|^2\,dx\\
  &=(2\pi)^{-d}|\Omega||p|^2,
\end{split}
\end{equation}
from which the sharp lower bound \eqref{Buckling-ev-sum-sharp-lower-bound} follows immediately.

\begin{rem}
From \eqref{sharp-FT-upper-bound} it is possible to deduce sharp lower bounds for sums of eigenvalues of higher order buckling problems  $(-\Delta u)^m=-\sigma\Delta u$.

\end{rem}

The sharp lower bound \eqref{Buckling-ev-sum-sharp-lower-bound} is equivalent to a sharp upper bound for $R_1(z)$. Namely, we have the following.

\begin{thm}\label{R1-upper-thm}
Let $\Omega$ be a domain in $\mathbb R^d$ of finite measure. Then
\begin{equation}\label{Buckling-ev-R-1-sharp-upper-bound}
    R_1(z)\leq \frac{2}{d+2}\,(2\pi)^{-d}B_d|\Omega|z^{1+\frac{d}{2}}.
  \end{equation}
\end{thm}
\begin{proof}
For $w\geq 0$, let $$
\mathcal L[f](w):=\sup_{z\geq 0}(z w-f(z))
$$
denote the Legendre transform of a convex, non-negative function $f$ defined on $[0,+\infty)$. We recall that, for $z,w\geq 0$
$$
f(z)\leq g(z)\iff\mathcal L[f](w)\geq \mathcal L[g](w)
$$
Setting $f(z)=R_1(z)$ and $g(z)=\frac{2}{d+2}\,(2\pi)^{-d}B_d|\Omega|z^{1+\frac{d}{2}}$, we have
$$
\mathcal L[f](w)=(w-[w])\sigma_{[w]+1}+\sum_{j=1}^{[w]}\sigma_j,
$$
and
$$
\mathcal L[g](w)=\frac{d}{d+2}C_d w\left(\frac{w}{|\Omega|}\right)^{\frac{2}{d}}.
$$
When $w\in\mathbb N$, then $\mathcal L[f](w)\geq \mathcal L[g](w)$ is exactly \eqref{Buckling-ev-sum-sharp-lower-bound}. Since $\mathcal L[f](w),\mathcal L[g](w)$ are convex, the inequality holds for all $w\geq 0$. This implies that $f(z)\leq g(z)$ for all $z\geq 0$. For more discussions on the equivalence of bounds on Riesz means and averages we refer to \cite{laptev_weidl}.
\end{proof}

\begin{rem}%\label{comparison}
We note that \eqref{Buckling-ev-R-1-sharp-upper-bound} corresponds to the classical Berezin-Li-Yau lower bound when we consider $R_1(z)$ to be the first Riesz mean for Dirichlet Laplacian eigenvalues. Its validity for buckling eigenvalues can be deduced alternatively from the fact that
$$
\lambda_j\leq\Lambda_j^{\frac{1}{2}}\leq\sigma_j.
$$
We note that this chain of inequalities has been proved in \cite{liu} for domains of class $C^2$ with strict inequalities. However it is not difficult to prove its validity under minimal assumptions on $\Omega$, that is, $\Omega$ of finite measure. In fact we just need that the existence of the discrete spectrum and its variational characterization are ensured. For example, in order to prove $\lambda_j\leq\sigma_j$ we just note that for any $0\ne u\in H^2_0(\Omega)$, from Cauchy-Schwarz inequality we have
$$
\left(\int_{\Omega}|\nabla u|^2dx\right)^2\le\left(\int_{\Omega} u^2dx\right)\left(\int_{\Omega} (\Delta u)^2dx\right).
$$
From this, we get
$$
\frac{\int_{\Omega} |\nabla u|^2dx}{\int_{\Omega} u^2dx}\le\frac{\int_{\Omega} (\Delta u)^2dx}{\int_{\Omega} |\nabla u|^2dx}
$$
for all $u\in H^2_0(\Omega)$. Thus, by the above inequality we deduce that
\begin{multline*}
\lambda_j=\min_{\substack{U\subset H^1_0(\Omega)\\{\rm dim}\ U=j}}\max_{u\in U\setminus\{0\}}\frac{\int_{\Omega} |\nabla u|^2dx}{\int_{\Omega} u^2dx}\leq \inf_{\substack{U\subset H^2_0(\Omega)\\{\rm dim}\ U=j}}\max_{u\in U\setminus\{0\}}\frac{\int_{\Omega} |\nabla u|^2dx}{\int_{\Omega} u^2dx}\\
\leq \min_{\substack{U\subset H^2_0(\Omega)\\{\rm dim}\ U=j}}\max_{u\in U\setminus\{0\}}\frac{\int_{\Omega} (\Delta u)^2dx}{\int_{\Omega} |\nabla u|^2dx}=\sigma_j.
\end{multline*}
We remark that when the domain is regular enough ($C^2$ or convex), the inequalities are strict, otherwise we would find an eigenfunction associated with an eigenvalue of the Dirichlet Laplacian $\lambda_j$ belonging to $H^2_0(\Omega)$, which is impossible.

\begin{rem}
We also cite here the following inequality due to Payne \cite{payne}
\begin{equation}\label{payne_0}
\Lambda_1\ge\lambda_1\sigma_1.
\end{equation}
This inequality is obtained from a similar comparison of the Rayleigh quotients and is therefore valid for any domain $\Omega\subset\mathbb R^d$ of finite measure. Moreover, it can be generalized to
\begin{equation}
\label{generalized_payne}
\Lambda_j\ge\max\{\lambda_1\sigma_j,\lambda_j\sigma_1\}.
\end{equation}
The variational characterization of $\Lambda_j,\lambda_j$, and $\sigma_j$ does not allow for any improvement of \eqref{generalized_payne}. However, in the one-dimensional case we establish the analogue of \eqref{payne_0} for all $j\in\mathbb N$, see Proposition \ref{generalized_payne_thm}.
\end{rem}

\end{rem}

\subsection{Proof of Weyl's law for buckling eigenvalues}

Theorem \ref{weyl_buckling} is now a consequence of Theorems \ref{R2-lower-thm} and \ref{R1-upper-thm}.
 
\begin{proof}[Proof of Theorem \ref{weyl_buckling}] 
From inequality \eqref{Buckling-ev-R-1-sharp-upper-bound} and the fact that $R_2(z)=2\int_{0}^{z}R_1(t)\,dt$, we deduce for all $z>0$
  \begin{equation*}%\label{Buckling-ev-R-2-sharp-upper-bound}
    R_2(z)\leq \frac{8}{(d+2)(d+4)}\,(2\pi)^{-d}B_d|\Omega|z^{2+\frac{d}{2}},
  \end{equation*}
  showing that
  \begin{equation*}
     \limsup_{z\rightarrow+\infty}R_2(z)z^{-2-\frac{d}{2}}\leq\frac{8}{(d+2)(d+4)}\,(2\pi)^{-d}B_d|\Omega|.
  \end{equation*}
Combining this with \eqref{R2-lower}  implies \eqref{Buckling-ev-Weyl-R-2}.
  
Weyl's law for the counting function \eqref{Buckling-ev-Weyl-R-0} follows from \eqref{Buckling-ev-Weyl-R-2} by applying twice Lemma \ref{technical} here below, since
  $\displaystyle R_2(z)=2\int_{0}^{z}R_1(t)\,dt$ and  $\displaystyle R_1(z)=\int_{0}^{z}N(t)\,dt$.

\end{proof}

We prove now the technical lemma used in the proof of Theorem \ref{weyl_buckling}. An alternative formulation may be found in \cite[Lemma 3]{geisinger}.
\begin{lem}\label{technical}
  Let $f:[0,+\infty)\rightarrow \mathbb{R}$ be an increasing function and $\displaystyle F(z)=\int_{0}^{z}f(t)\,dt$. If
  \begin{equation}\label{F-limit}
    \lim_{z\rightarrow+\infty}z^{-1-p}F(z)=1
  \end{equation}
  for some $p>0$, then
  \begin{equation*}%\label{f-limit}
    \lim_{z\rightarrow+\infty}z^{-p}f(z)=p+1.
  \end{equation*}
\end{lem}
\begin{proof}
  Since $f$ is increasing, for any $z>0$ and $h>0$
  \begin{equation}\label{F-f-ineq-1}
    F(z+h)-F(z)=\int_{z}^{z+h}f(t)\,dt\geq h f(z).
  \end{equation}
  By the definition of the limit \eqref{F-limit}, for any $\epsilon>0$ there exists $z_{\epsilon} \geq 0$ such that $\displaystyle |z^{-1-p}F(z)-1|<\epsilon$ for all $z>z_{\epsilon}$. Therefore for all $h>0$, $z>z_{\epsilon}$ inequality \eqref{F-f-ineq-1} implies
  \begin{equation*}
  \begin{split}
    f(z)&\leq (1+\epsilon)\,\frac{(z+h)^{p+1}-z^{p+1}}{h}+\frac{2\epsilon z^{p+1}}{h}\\
   &\leq (1+\epsilon)(p+1)(z+h)^{p}+\frac{2\epsilon z^{p+1}}{h}.\\
  \end{split}
  \end{equation*}
  We choose $h=\sqrt{\epsilon}\,z$. Then for all $z>z_{\epsilon}$
  \begin{equation*}
    z^{-p}f(z)\leq (p+1)(1+\epsilon)(1+\sqrt{\epsilon})^{p}+2\sqrt{\epsilon}
  \end{equation*}
  and hence
  \begin{equation}\label{lim-sup-f}
    \limsup_{z\rightarrow+\infty}z^{-p}f(z)\leq p+1.
  \end{equation}
  Similarly, for all $0<h<z$
   \begin{equation*}%\label{F-f-ineq-2}
    F(z)-F(z-h)=\int_{z-h}^{z}f(t)\,dt\leq h f(z),
  \end{equation*}
and therefore given $\epsilon>0$ for any $z>z-h>z_{\epsilon}$,
\begin{equation*}
  \begin{split}
    f(z)&\geq (1+\epsilon)\,\frac{z^{p+1}-(z-h)^{p+1}}{h}-\frac{2\epsilon z^{p+1}}{h}\\
   &\geq (1+\epsilon)(p+1)(z-h)^{p}-\frac{2\epsilon z^{p+1}}{h}.\\
  \end{split}
  \end{equation*}
  As before we choose $h=\sqrt{\epsilon}\,z$. Then for any $0<\epsilon<\frac{1}{4}$ we have $\displaystyle z-h >\frac{z}{2}$ and hence for all $z>2z_{\epsilon}$,
   \begin{equation*}
    z^{-p}f(z)\geq (p+1)(1+\epsilon)(1-\sqrt{\epsilon})^{p}-2\sqrt{\epsilon}
  \end{equation*}
  and hence
  \begin{equation*}%\label{lim-inf-f}
    \liminf_{z\rightarrow+\infty}z^{-p}f(z)\geq p+1
  \end{equation*}
  which, together with \eqref{lim-sup-f}, proves the lemma.
\end{proof}

%%%%%%%%%%%%%%%%%%%%%%%%%%%%%%%%%%%%%%%%%%%%%%%%%%%%%%%%%%%%%%%
%%%%%%%%%%%%%%%%%%%%%%%%%%%%%%%%%%%%%%%%%%%%%%%%%%%%%%%%%%%%%%%
%%%%%%%%%%%%%%%%%%%%%   second term
%%%%%%%%%%%%%%%%%%%%%%%%%%%%%%%%%%%%%%%%%%%%%%%%%%%%%%%%%%%%%%%
%%%%%%%%%%%%%%%%%%%%%%%%%%%%%%%%%%%%%%%%%%%%%%%%%%%%%%%%%%%%%%%
%

\section{Two term asymptotics for the buckling problem on balls}\label{sec:4}
The aim of this section is to prove Theorem \ref{conj_ball}, namely the validity of the two-terms asymptotic expansion stated in Conjecture \ref{conj_weyl} for an open ball in $\mathbb R^d$. Since buckling eigenvalues behave like the eigenvalues of the Laplacian under scaling, it is sufficient to prove Theorem \ref{conj_ball}  for the unit ball in $\mathbb R^d$ centered in zero.

Before proceeding with this analysis, we believe that it is worth recalling the formal arguments which allow us to state Conjecture \ref{conj_weyl}. In the case of the biharmonic operator (with Dirichlet, Neumann, Navier boundary conditions) two-terms asymptotics for the counting function have been computed in \cite{bps} by applying the arguments of \cite{safvas}. On the other hand, the techniques in \cite{safvas} do not apply, in principle, to the case of the eigenvalues of an operator pencil. Nevertheless it is possible, at least formally, to exploit the arguments of \cite{safvas} also in this situation. This leads to formula \eqref{buckling-two-term-asympt-N-of-z} which is then a reasonable ansatz for a two-terms expansion. 

In analogy with \cite[Theorem 1.6.1]{safvas} and \cite[Theorem 3.2]{bps}, we write
\begin{equation}\label{ansatz}
N(z)=c_0z^{\frac{d}{2}}-c_1z^{\frac{d-1}{2}}+o\left( z^{\frac{d-1}{2}}\right),
\end{equation}
where
$$
c_0=(2\pi)^{-d}\int_{T^*\Omega}\chi_{\{A(x,\xi)^2\leq A(x,\xi)\}}dxd\xi\,,\ \ \ c_1=(2\pi)^{1-d}\int_{T^*\partial\Omega}{\rm shift}^+(1,x',\xi')dx'd\xi'.
$$
Here $(x,\xi)$ are the elements of the cotangent bundle $T^*\Omega$, $(x',\xi')$ are the elements of the cotangent bundle $T^*\partial\Omega$, and $A(x,\xi)$ is the principal symbol associated with the operator, which, in the buckling case, corresponds to $|\xi|^2$. Therefore we immediately have $c_0=(2\pi)^{-d}B_d|\Omega|$. On the other hand, the function ${\rm shift}^+$ appearing in the formula for $c_1$ is the so-called spectral shift function associated with the auxiliary problem
\begin{equation}\label{auxiliary}
\begin{cases}
T_{\xi'}^2v(x_1)=\eta T_{\xi'}v(x_1)\,, & x_1\in[0,+\infty),\\
v(0)=v'(0)=0,
\end{cases}
\end{equation}
where $T_{\xi'}=|\xi'|^2-\frac{d^2}{dx_1^2}$. We refer to \cite[Chapter 1.6]{safvas} for the precise definition of the spectral shift function.  Roughly speaking, the auxiliary problem \eqref{auxiliary} is obtained from \eqref{buckling} by locally flattening the boundary in such a way that, locally, $\Omega=\{(x_1,...,x_d):x_1>0\}$ and $\partial\Omega=\{(0,x_2,...,x_d)\}$, and by taking the Fourier transform with respect to $x'=(x_2,...,x_d)$. Generalized eigenfunctions of \eqref{auxiliary} for $\eta\geq|\xi'|^2$ are of the form $a_1^- e^{i\zeta_1^- x_1}+a_1^+ e^{i\zeta_1^+ x_1}+a_2^+ e^{i\zeta_2^+ x_1}$, where $\zeta_1^{\pm}=\pm\sqrt{\eta-|\xi'|^2}$, $\zeta_2^+=i|\xi'|$. On the other hand, problem \eqref{auxiliary} has no solution for $\eta<|\xi'|^2$ compatible with the boundary conditions. Imposing the boundary conditions in \eqref{auxiliary} to a generalized eigenfunction we obtain the explicit expressions for $a_1^+$ and $a_1^-$, which we omit here. According to \cite[Chapter 1.6]{safvas}, it turns out that the ${\rm shift}^+$ function can be written as
$$
{\rm shift}^+(\eta,x',\xi')=\frac{\arg_0\left(i\frac{a_1^+}{a_1^-}\right)}{2\pi},
$$
where $\arg_0$ is a the only branch of the complex argument that satisfies
\begin{equation}\label{auxiliary1}
\lim_{\eta\rightarrow |\xi'|^2}\left|\arg_0\left(i\frac{a_1^+}{a_1^-}\right)\right|=\frac{\pi}{2}.
\end{equation}
Finally, from condition  \eqref{auxiliary1} and from the explicit expressions of $a_1^+,a_1^-$, we get
\begin{equation}\label{shift}
{\rm shift}^+(\eta,x',\xi')=-\frac{3}{2}\pi+2\arcsin\left(\frac{|\xi'|}{\sqrt{\eta}}\right).
\end{equation}
Now formula \eqref{buckling-two-term-asympt-N-of-z} follows from \eqref{ansatz}, \eqref{shift}, and standard computations. Interested readers may refer to \cite[Chapter 1.6]{safvas} for more details, and to \cite{bps} for the corresponding computations in the case of the biharmonic operator.

\subsection{Buckling eigenvalues and eigenfunctions of the ball}
We recall that an eigenfunction of problem \eqref{buckling} on the unit ball of $\mathbb R^d$ can be written in spherical coordinates $(r,\theta)$, $\theta=(\theta_1,...,\theta_{n-1})$, as a product of a radial part and an angular part. In particular, any eigenfunction has the form
\begin{equation}\label{generic_eigenfunction_ball}
\left(j_l(\sqrt{\sigma})r^l-j_l(\sqrt{\sigma}r)\right)H_l(\theta),
\end{equation}
for some $l\in\mathbb N\cup\{0\}$. Here $H_l(\theta)$ is a spherical harmonic of degree $l$ in $\mathbb R^d$, while $j_l$ denotes the ultraspherical Bessel function of the first kind of order $l$ in $\mathbb R^d$, namely
$$
j_l(z)=z^{1-\frac{d}{2}}J_{\frac{d}{2}-1+l}(z),
$$
where $J_{\nu}(z)$ denotes the standard Bessel function of the first kind of order $\nu$. The proof that any eigenfunction of \eqref{buckling} on the unit ball is of the form \eqref{generic_eigenfunction_ball} is standard, we may refer e.g., to \cite{buosoprovenzano} and references therein.

We note that, from \eqref{generic_eigenfunction_ball}, it immediately follows that the eigenfunctions vanish at the boundary, therefore they already satisfy one of the boundary conditions in \eqref{buckling}. Thus the eigenvalues are determined, for each given $l\in\mathbb N\cup\{0\}$, by imposing the other boundary condition, which is equivalent to the equation
\begin{equation}\label{implicit0}
j_l(\sqrt{\sigma})l-\sqrt{\sigma}j_l'(\sqrt{\sigma})=0.
\end{equation}
From the definition of $j_l$ and from the recurrence relation $J_{\nu}'(z)=J_{\nu+1}(z)+\frac{\nu}{z}J_{\nu}(z)$ (see e.g., \cite{abram}) we immediately verify that \eqref{implicit0} is equivalent to
\begin{equation*}%\label{implicit}
J_{\frac{d}{2}+l}(\sqrt{\sigma})=0.
\end{equation*}
It is also worth recalling that the eigenvalues of the Dirichlet Laplacian on the unit ball in $\mathbb R^d$ are the zeros of the equation
\begin{equation*}%\label{implicit_dir}
J_{\frac{d}{2}+l-1}(\sqrt{\lambda})=0.
\end{equation*}
In both cases, the multiplicity of an eigenvalue corresponding to angular momentum $l\in\mathbb N\cup\{0\}$ is given by
\begin{equation*}%\label{radial-ev-multiplicity}
  M_{l,d}:=\binom{l+d-1}{d-1}-\binom{l+d-3}{d-1}=\frac{(2l+d-2)\cdot (l+d-3)!}{l! (d-2)!},
\end{equation*}
which is the dimension of the space of spherical harmonics of order $l$ in $\mathbb R^d$. Let us also recall the identity
\begin{equation}\label{radial-ev-multiplicity-2}
  M_{l,d}=\frac{(2l+d-2)}{l}\,\binom{l+d-3}{d-2}
\end{equation}
and the addition formula
\begin{equation}\label{multiplicity-addition-formula}
  M_{l+1,d+1}=M_{l,d+1}+M_{l+1,d}.
\end{equation}

\subsection{Zeros of Bessel functions and radial eigenvalues}
Let $\sigma_{d,l,n}$ denote be the $n$-th radial eigenvalue corresponding to angular momentum $l$ for the buckling problem on the $d$-dimensional unit ball, namely
\begin{equation*}%\label{radial-eigenvalues-Bessel-fct-zeros_buckling}
  \sigma_{d,l,n}=x_{l+\frac{d}{2},n}^2,
\end{equation*}
where $x_{l+\frac{d}{2},n}$ denotes the $n$-th zero of the Bessel function $J_{l+\frac{d}{2}}$, $n\in\mathbb N$. The number of nodes of the radial part of the corresponding eigenfunction is $n-1$. Similarly, by $\lambda_{d,l,n}$ we denote the $n$-th radial eigenvalue corresponding to angular momentum $l$ of the Dirichlet Laplacian on the $d$-dimensional unit ball, namely
\begin{equation*}%\label{radial-eigenvalues-Bessel-fct-zeros_laplacian}
\lambda_{d,l,n}=x_{l-1+\frac{d}{2},n}^2,
\end{equation*}
Note that $\sigma_{d,l,n}$ and $\lambda_{d,l,n}$  have the same multiplicities. We also have the relation
\begin{equation}\label{radial-eigenvalues-buckling-Dirichlet-relation}
  \sigma_{d,l,n}=\lambda_{d,l+1,n}=x_{l+\frac{d}{2},n}^2.
\end{equation}
However, here multiplicities are $M_{l,d}$ for the buckling eigenvalues and $M_{l+1,d}$ for the Dirichlet Laplacian eigenvalues. When $n$ is large the zeros $x_{l+\frac{d}{2},n}$ admit the asymptotic expansion (see \cite{abram,GR})
\begin{equation}\label{Zeros-expansion}
  x_{l+\frac{d}{2},n}= \frac{\pi(4n+2l+d-1)}{4}-\frac{4\left(l+\frac{d}{2}\right)^2-1}{2\pi(4n+2l+d-1)}+O(n^{-3}).
\end{equation}
\subsection{Counting functions}
Throughout this section, for $z>0$ we denote by $N^B_d(z)$ the number of buckling eigenvalues below $z$. For any $l\in\mathbb N\cup\{0\}$ fixed we denote by $N^B_{d,l}(z)$ the number of radial buckling eigenvalues $\sigma_{d,l,n}$ below $z$. By definition,

\begin{equation*}%\label{counting-function-buckling-sum-of-angular-momentum-sectors}
  N^B_d(z)=\sum_{l=0}^{\infty} M_{l,d}\,N^B_{d,l}(z).
\end{equation*}
Note that the above sum is always finite since there is an index $l_{\max}$ depending on $z$ and $d$ such that $N^B_{d,l}(z)=0$ for all $l> l_{\max}$. We define $N^D_d(z)$, $N^D_{d,l}(z)$ for the counting functions of the Dirichlet Laplacian in a similar way, and hence
\begin{equation*}%\label{counting-function-Laplacian-sum-of-angular-momentum-sectors}
  N^D_d(z)=\sum_{l=0}^{\infty} M_{l,d}\,N^D_{d,l}(z).
\end{equation*}
For any $l\in\mathbb N\cup\{0\}$ fixed, let $N^J_{d,l}(\zeta)$ be the number of zeros of the Bessel function $J_{l+\frac{d}{2}}$ below $\zeta$. Then $N^J_{d+2,l}(\zeta)=N^J_{d,l+1}(\zeta)$ and by \eqref{radial-eigenvalues-buckling-Dirichlet-relation}
\begin{equation*}%\label{D-B-counting-function-relation-of-angular-momentum-sectors}
  N^B_{d,l}(z)=N^J_{d,l}(z^{\frac{1}{2}})= N^D_{d,l+1}(z).
\end{equation*}
\subsection{Asymptotic Expansions of Counting functions}
We recall that the counting function $ N^D_d(z)$ for the Dirichlet Laplacian on the unit ball admits the following asymptotic expansion as $z$ tends to $+\infty$:

\begin{equation}\label{Dirichlet-Laplacian-two-term-asympt-N-of-z-unit-ball}
  N^D_d(z)=\frac{1}{\Gamma(\frac{d}{2}+1)^2}\,\left(\frac{z}{4}\right)^{\frac{d}{2}}
  -\frac{d\pi^{\frac{1}{2}}}{4\Gamma(\frac{d}{2}+1)\Gamma(\frac{d+1}{2})}\,\left(\frac{z}{4}\right)^{\frac{d-1}{2}}+o\left(z^{\frac{d}{2}-\frac{1}{2}}\right).
\end{equation}
Therefore Conjecture \ref{conj_weyl} for the counting function of the buckling eigenvalues of the unit ball is verified provided that
\begin{equation}\label{Buckling-two-term-asympt-N-of-z-unit-ball}
 N^B_d(z)- N^D_d(z)=-(2\pi)^{1-d}B_{d-1}^2z^{\frac{d-1}{2}}+o\left(z^{\frac{d}{2}-\frac{1}{2}}\right).
\end{equation}
Note that the quantity on the right-hand side of \eqref{Buckling-two-term-asympt-N-of-z-unit-ball} corresponds precisely to the leading term of the asymptotic expansion for $N^D_{d-1}(z)$ and $N^B_{d-1}(z)$ on the $(d-1)$-dimensional unit ball. 
\subsection{Proof of two-terms asymptotics for the unit disc}
We provide first the proof of Theorem \ref{conj_ball} in the case $d=2$, which has to be treated separately.
\begin{thm}
  For the unit disk in $\mathbb{R}^2$ we have the following asymptotic expansion as $z$ tends to $+\infty$:
  \begin{equation*}%\label{Buckling-two-term-asympt-N-of-z-unit-disc-1}
    N^B_2(z)- N^D_2(z)=-\frac{2}{\pi}\,z^{\frac{1}{2}}+o\left(z^{\frac{1}{2}}\right),
  \end{equation*}
  or, equivalently,
  \begin{equation*}%\label{Buckling-two-term-asympt-N-of-z-unit-disc-2}
    N^B_2(z)=\frac{1}{4}\,z-\left(\frac{1}{2} +\frac{2}{\pi}\right)\,z^{\frac{1}{2}}+o\left(z^{\frac{1}{2}}\right).
  \end{equation*}
\end{thm}
\begin{proof}
  For any $l\geq 1$ the multiplicities $M_{l,2}=2$ do not depend on $l$. By \eqref{radial-eigenvalues-buckling-Dirichlet-relation} we have
  $\sigma_{d,l,n}=\lambda_{d,l+1,n}$ for all $l\in\mathbb N\cup\{0\}$ and therefore
  \begin{equation*}
    \begin{split}
       N^B_2(z) & =N^B_{2,0}(z)+\sum_{l=1}^{\infty}2 \,N^B_{2,l}(z)\\
         & =N^D_{2,1}(z)+\sum_{l=1}^{\infty}2 \,N^D_{2,l+1}(z)\\
         &= -N^D_{2,1}(z)+\sum_{l=1}^{\infty}2 \,N^D_{2,l}(z)\\
         &= N^D_2(z) - N^D_{2,0}(z)-N^D_{2,1}(z).\\
    \end{split}
  \end{equation*}
  According to the asymptotics \eqref{Zeros-expansion} for the zeros of Bessel functions the leading terms for the counting functions $N^D_{2,0}(z)$ and $ N^D_{2,1}$ are $\frac{z^{\frac{1}{2}}}{\pi}$. This proves the theorem.
\end{proof}
\subsection{Proof of two-terms asymptotics for the unit ball}
We prove here Theorem \ref{conj_ball} for $d\geq 3$.
\begin{thm}
  For the unit ball in $\mathbb{R}^d$, $d\geq 3$, we have the asymptotic expansion as $z$ tends to $+\infty$:
  \begin{equation}\label{Buckling-two-term-asympt-N-of-z-unit-ball-1}
    N^B_d(z)- N^D_d(z)= -N^D_{d-1}(z)+o\left(z^{\frac{d}{2}-\frac{1}{2}}\right).
  \end{equation}
\end{thm}
\begin{proof}
   By \eqref{radial-eigenvalues-buckling-Dirichlet-relation} we have
  $\sigma_{d,l,n}=\lambda_{d,l+1,n}$ for all $l\in\mathbb N\cup\{0\}$ and therefore
  \begin{equation*}
    \begin{split}
       N^B_d(z) & =\sum_{l=0}^{\infty}  M_{l,d}\,N^B_{d,l}(z)\\
         & =\sum_{l=0}^{\infty} M_{l,d}\,N^D_{d,l+1}(z)\\
         &= \sum_{l=0}^{\infty} M_{l+1,d}\,N^D_{d,l+1}(z)+\sum_{l=0}^{\infty}( M_{l,d}-M_{l+1,d})\,N^D_{d,l+1}(z)\\
         &= N^D_d(z) - N^D_{d,0}(z)+\sum_{l=0}^{\infty}( M_{l,d}-M_{l+1,d})\,N^D_{d,l+1}(z).\\
    \end{split}
  \end{equation*}
  By the addition formula \eqref{multiplicity-addition-formula} we have: $M_{l,d}-M_{l+1,d}=-M_{l+1,d-1}$. Hence
  \begin{equation}\label{step}
    \begin{split}
       N^B_d(z)-N^D_d(z) & =- N^D_{d,0}(z)-\sum_{l=0}^{\infty}M_{l+1,d-1}\,N^D_{d,l+1}(z)\\
         & =-\sum_{l=0}^{\infty} M_{l,d-1}\,N^D_{d,l}(z)\\
         &= -N^D_{d-1}(z)-\sum_{l=0}^{\infty} M_{l,d-1}\,(N^D_{d,l}(z)-N^D_{d-1,l}(z)).\\
    \end{split}
  \end{equation}
  We are left to show that $\sum_{l=0}^{\infty} M_{l,d-1}\,(N^D_{d,l}(z)-N^D_{d-1,l}(z))=o\left(z^{\frac{d}{2}-\frac{1}{2}}\right)$. First, from \eqref{radial-ev-multiplicity-2} we deduce that there is a positive constant $C$ independent of $l$ and $d$ such that $M_{l,d-1}\leq C\,l^{d-3}$. Moreover
 $$
 \sum_{l=0}^{\infty} M_{l,d-1}\,(N^D_{d,l}(z)-N^D_{d-1,l}(z))=\sum_{l=0}^{l_{\max}} M_{l,d-1}\,(N^D_{d,l}(z)-N^D_{d-1,l}(z)).
 $$
  Now we claim that there exists a positive constant $C'$ such that $l_{\max}\leq C' z^{\frac{1}{2}}$. To do so, we note that $F_{\nu}(x):=J_{\nu}(Ex)$ with $E>0$ satisfies the ordinary differential equation
  $$
  -F_{\nu}''(x)-\frac{1}{x}F_{\nu}'(x)+\frac{\nu^2}{x^2}F_{\nu}(x)=E^2 F_{\nu}(x).
  $$
  Multiplying both sides by $xF_{\nu}(x)$, assuming that $F_{\nu}(1)=J_{\nu}(E)=0$, and integrating over $(0,1)$ by parts we get
    $$
 \int_0^1 x\left((F_{\nu}'(x))^2+\frac{\nu^2}{x^2}F_{\nu}^2(x)\right)x dx=E^2\int_0^1 (F_{\nu}(x))^2xdx.
  $$
  Therefore $\nu^2\leq E^2$. Hence the first zero of $J_{\nu}$ is lower bounded by $\nu$. In our situation $\nu=l+\frac{d}{2}-1$ and $E=\lambda_{d-1,l,1}$, thus $N^D_{d-1,l}(z)=N^D_{d,l}(z)=0$ for $l\geq z^{\frac{1}{2}}-\frac{d-1}{2}+1$, proving the claim. 

Now we recall that the real positive zeros of $J_{\nu}$ and $J_{\nu+\varepsilon}$ are interlaced as long as $0<\varepsilon<1$ (see e.g., \cite{interlacing}), therefore $N^D_{d,l}(z)-N^D_{d-1,l}(z)\leq 1$. Then $\sum_{l=0}^{\infty} M_{l,d-1}\,(N^D_{d,l}(z)-N^D_{d-1,l}(z))=O(z^{\frac{d}{2}-1})$. Using this asymptotic relation  and formula \eqref{Dirichlet-Laplacian-two-term-asympt-N-of-z-unit-ball} for $N^D_{d-1}(z)$ in \eqref{step} we get \eqref{Buckling-two-term-asympt-N-of-z-unit-ball-1}. This concludes the proof. 
\end{proof}

 %----------------------------------------------------------------

%%%%%%%%%%%%%%%%%%%%%%%%%%%%%%%%%%%%%%%%%%%%%%%%%%%%%%%%%%%%%%%
%%%%%%%%%%%%%%%%%%%%%%%%%%%%%%%%%%%%%%%%%%%%%%%%%%%%%%%%%%%%%%%
%%%%%%%%%%%%%%%%%%%%%   one dim
%%%%%%%%%%%%%%%%%%%%%%%%%%%%%%%%%%%%%%%%%%%%%%%%%%%%%%%%%%%%%%%
%%%%%%%%%%%%%%%%%%%%%%%%%%%%%%%%%%%%%%%%%%%%%%%%%%%%%%%%%%%%%%%
% ----------------------------------------------------------------

\section{The buckling problem in one dimension}\label{sec:3}
When $d=1$ problem \eqref{buckling} in the open interval $(0,1)$ reads
\begin{equation}\label{1d-buckling-ev-problems}
\left\{\begin{array}{l}
u''''(x)=- \sigma u''(x), \quad x\in(0,1),\\
u(0)=u'(0)=u(1)=u'(1)=0. \\
\end{array}\right.
\end{equation}
Problem \eqref{1d-buckling-ev-problems} admits an increasing sequence of simple and positive eigenvalues $\{\sigma_j\}_j$ diverging to plus infinity, and they can be characterized through the minimax procedure exactly as in the higher dimensional case (see \eqref{buckling_minimax}). In particular, the eigenfunctions are of the form
\begin{equation*}
  u_j(x)=A\big(\cos(\gamma_jx)-1\big)+ \sin(\gamma_jx)-\gamma_jx,
\end{equation*}
where $  A=\frac{\sin(\gamma_j)-\gamma_j}{1-\cos(\gamma_j)}$ and
$\gamma_j$ is the $j$-th positive solution of the equation
\begin{equation*}
  2(1-\cos\gamma_j)-\gamma_j\sin(\gamma_j)=0,
\end{equation*}
that can be split into two different equations, namely $ \sin(\frac{\gamma_j}{2})=0$ or $  \tan(\frac{\gamma_j}{2})=\frac{\gamma_j}{2}$. The eigenvalues are given by
\begin{equation*}%\label{1d-buckling-ev}
  \sigma_j= \gamma_j^2.
\end{equation*}
Consequently, for positive integers $j$, the eigenvalues of the buckling problem are of the form
\begin{equation}\label{1d-buckling-ev-explicit}
  \sigma_j=\bigg(\pi(j+1)-t_j\bigg)^2
\end{equation}
where $t_j=0$ if $j$ is odd and $t_j\in(0,\pi)$ is the first positive root of the equation
\begin{equation}\label{t-n-equation}
  \sin \left(\frac{t}{2}\right)- \frac{2\cos(\frac{t}{2})}{\pi(j+1)-t}=0
\end{equation}
if $j$ is even. From \eqref{t-n-equation} we infer that $t_j=\frac{4}{\pi(j+1)}+o(\frac{1}{j})$ as $j$ tends to infinity, when $j$ is even. We are now ready to prove Theorem \ref{conj_1d}. Since buckling eigenvalues behave like Laplacian eigenvalues under scaling it is sufficient to prove Theorem \eqref{conj_1d} in the case $L=1$.

\begin{proof}[Proof of Theorem \ref{conj_1d}].
Let $L=1$. Then the buckling eigenvalues are given by \eqref{1d-buckling-ev-explicit}. We note that for any positive integer $k$
\begin{equation*}
  \sum_{j=1}^{k}z-\bigg(\pi(j+1)-t_j\bigg)^2=\sum_{j=1}^{k}z-\pi^2(j+1)^2+2\pi(j+1)t_j-t_j^2.
\end{equation*}
The last two terms of the sum will just bring a contribution of order $k$ as $k$ gets large. Therefore, choosing $k=\left[\frac{z^{\frac{1}{2}}-1}{\pi}\right]$, from
\begin{equation*}
  \sum_{j=1}^{k}z-\pi^2(j+1)^2=kz-\frac{\pi^2}{3}k^3-\frac{3\pi^2}{2}k^2-\frac{13\pi^2}{6}k
\end{equation*}
we get
\begin{equation*}
  R_1(z)= \frac{2}{3\pi}\,z^{\frac{3}{2}}-\frac{3z}{2}+O(z^{\frac{1}{2}}).
\end{equation*}
proving the claim.
\end{proof}

We include in this section another inequality relating $\sigma_j$ with the eigenvalues $\lambda_j$ and $\Lambda_j$ of problems \eqref{dirichlet_laplacian} and \eqref{dirichlet_bilaplacian} on $(0,1)$, which is interesting by itself.

\begin{prop}\label{generalized_payne_thm}
Let $\sigma_j,\lambda_j$ and $\Lambda_j$ be the eigenvalues of \eqref{buckling}, \eqref{dirichlet_laplacian} and \eqref{dirichlet_bilaplacian} on $(0,1)$, respectively. Then, for all $j\in\mathbb N$
\begin{equation*}
%\label{generalized_payne_2}
\Lambda_j>\lambda_j\sigma_j.
\end{equation*}
\end{prop}
\begin{proof}
We recall from \cite[Proposition A.1]{bps} that the biharmonic Dirichlet eigenvalues on $(0,1)$ satisfy
\begin{equation}\label{1d-dir-ev-explicit}
\Lambda_j=\left(\pi\left(j+\frac{1}{2}\right)-(-1)^js_j\right)^4\,,\ \ \ 0<s_j<\frac{\pi}{2}.
\end{equation}
The Dirichlet eigenvalues of the Laplacian on $(0,1)$ are given by
$$
\lambda_j=\pi^2j^2,
$$
and the buckling eigenvalues are given by \eqref{1d-buckling-ev-explicit}. We distinguish the case $j$ odd and $j$ even.

\medskip
{\bf Odd $j$.} Assume that $j$ is odd. Then immediately we see that
$$
\Lambda_j-\lambda_j\sigma_j\geq \pi^4\left(j+\frac{1}{2}\right)^4-\pi^4j^2(j+1)^2=\frac{\pi^4}{16}\left(8j^2+8j+1\right)>0.
$$

\medskip
{\bf Even $j$.} Assume now that $j$ is even. Then
\begin{multline*}
\Lambda_j-\lambda_j\sigma_j=\left(\pi\left(j+\frac{1}{2}\right)-s_j\right)^4-\pi^2j^2\left(\pi(j+1)-t_j\right)^2\\
=\left(\pi j+\left(\frac{\pi}{2}-s_j\right)\right)^4-\pi^2j^2\left(\pi j+(\pi-t_j)\right)^2.
\end{multline*}
The right-hand side is strictly positive provided $s_j<\frac{t_j}{2}$. We claim that $s_j<\frac{t_j}{2}$ for all $j\in\mathbb N$. We recall from \cite[Proposition A.1]{bps} that
$$
\cos\left(\Lambda_j^{\frac{1}{4}}\right)\cosh\left(\Lambda_j^{\frac{1}{4}}\right)=1.
$$
This implies, together with \eqref{1d-dir-ev-explicit}, that
\begin{equation}\label{sj}
\sin^2(s_j)=\frac{1}{\cosh^2\left(\Lambda_j^{\frac{1}{4}}\right)}.
\end{equation}
Moreover, from \eqref{1d-buckling-ev-explicit} and \eqref{t-n-equation} we have
$$
\tan^2\left(\frac{t_j}{2}\right)=\frac{4}{\sigma_j},
$$
which implies
\begin{equation}\label{tj}
\sin^2\left(\frac{t_j}{2}\right)=\frac{4}{4+\sigma_j}.
\end{equation}
From \eqref{1d-buckling-ev-explicit}, \eqref{1d-dir-ev-explicit}, and from the fact that $0<s_j,\frac{t_j}{2}<\frac{\pi}{2}$, we deduce that $\sigma_j\leq 4\Lambda_j^{\frac{1}{2}}$. From \eqref{sj} and \eqref{tj} we deduce that $s_j<\frac{t_j}{2}$ if and only if
$$
\cosh^2\left(\Lambda_j^{\frac{1}{4}}\right)>1+\frac{\sigma_j}{4},
$$
which is verified being $4\Lambda_j^{\frac{1}{2}}\geq\sigma_j$. This proves the claim.
\end{proof}

We conclude this section with a final remark concerning another inequality by Payne, relating $\lambda_2$ and $\sigma_1$. 
\begin{equation}\label{payne2}
\lambda_2\leq\sigma_1.
\end{equation}
Inequality \eqref{payne2} is proved in \cite{payne} for $d=2$. However Payne's proof contained a gap, later filled in \cite{friedlander_rem} with a proof valid in any dimension. One-dimensional computations of this section show that in general one cannot expect that 
$$
\lambda_{j+1}\leq\sigma_j
$$
for $j\geq 2$. In fact, we have that, for $j$ odd,
$$
\lambda_{j+1}=\sigma_j,
$$
therefore the inequality becomes an equality. For $j$ even, we have
$$
\lambda_{j+1}>\sigma_j.
$$
Moreover, considering sums, we see that
$$
\sum_{j=1}^k\sigma_j<\sum_{j=1}^k\lambda_{j+1}
$$
therefore, for $k\geq 2$, Payne's inequality cannot be generalized to eigenvalue averages, at least for $d=1$.

We think that it is worth including here a simpler proof of \eqref{payne2}, valid in any dimension, and with the minimal assumptions on $\Omega$.

\begin{prop}
Inequality \eqref{payne2} holds for any domain $\Omega$ in $\mathbb R^d$ of finite measure.
\end{prop}
\begin{proof}
For any buckling eigenfunction $w_j$ associated with the eigenvalue $\sigma_j$ and with $\|\nabla w_j\|_2^2=1$, we have $1\leq\|w_j\|_2^2\sigma_j$. This is just a consequence of the Cauchy-Schwarz inequality applied to $1=\|\nabla w_j\|_2^2=-\int_{\Omega}w_j\Delta w_j dx$. Let now $u_1$ be an eigenfunction associated with $\lambda_1$, the first eigenvalue of the Dirichlet Laplacian on $\Omega$. If $\int_{\Omega}u_1 w_1 dx=0$, then clearly, from the min-max principle \eqref{minmax-dirichlet} it follows that
$$
\lambda_2\leq\frac{\|\nabla w_1\|_2^2}{\|w_1\|_2^2}\leq\sigma_1.
$$
If $w_1$ is not orthogonal to $u_1$, let
$$
\phi_i=\partial_iw_1+a_i w_1,
$$
for $i=1,...,d$. The coefficients $a_i$ are chosen in such a way that $\int_{\Omega}\phi_i u_i dx=0$. Moreover, $\phi_i\in H^1_0(\Omega)$. Then, from the min-max principle and integration by parts, we obtain
$$
\lambda_2\left(\|\partial_i w_1\|_2^2+|a_i|^2\|w_1\|_2^2\right)\leq\left(\|\nabla\partial_i w_1\|_2^2+|a_i|^2\|\nabla w_1\|_2^2\right).
$$
Summing over $i=1,...,d$ and defining $|a|^2:=\sum_{i=1}^d|a_i|^2$, we get
$$
\lambda_2\left(\|\nabla w_1\|_2^2+|a|^2\|w_1\|_2^2\right)\leq\sigma_1+|a|^2\|\nabla w_1\|_2^2
$$
which is rewritten as
$$
\lambda_2\leq\frac{\sigma_1+|a|^2}{1+|a|^2\|w_1\|_2^2}=\sigma_1-|a|^2\frac{\sigma_1\|w_1\|_2^2-1}{1+|a|^2\|w_1\|_2^2}\leq\sigma_1.
$$
\end{proof}
%%%%%%%%%%%%%%%%%%%%%%%%%%%%%%%%%%%%%%%%%%%%%%%%%%%%%%%%%%%%%%%
%%%%%%%%%%%%%%%%%%%%%%%%%%%%%%%%%%%%%%%%%%%%%%%%%%%%%%%%%%%%%%%
%%%%%%%%%%%%%%%%%%%%%   ACKNOWLEDGEMENTS
%%%%%%%%%%%%%%%%%%%%%%%%%%%%%%%%%%%%%%%%%%%%%%%%%%%%%%%%%%%%%%%
%%%%%%%%%%%%%%%%%%%%%%%%%%%%%%%%%%%%%%%%%%%%%%%%%%%%%%%%%%%%%%%
% ----------------------------------------------------------------

\section*{Acknowledgements}

The second and fourth authors acknowledge support of the SNSF project ``Bounds for the Neumann and Steklov eigenvalues of the biharmonic operator'', grant number 200021\_178736. 
%Most of the research in this paper was carried out while the first author held a post-doctoral position at \'Ecole Polytechnique F\'ed\'erale de Lausanne within the scope of this project.
%The second author expresses his gratitude to \'Ecole Polytechnique F\'ed\'erale de Lausanne for the hospitality that helped the development of this paper. 
The first and the second authors are members of the Gruppo Nazionale per l'Analisi Ma\-te\-ma\-ti\-ca, la Probabilit\`a e le loro Applicazioni (GNAMPA) of the I\-sti\-tuto Naziona\-le di Alta Matematica (INdAM). The third author is member of the Gruppo Nazionale per le Strutture Algebriche, Geometriche e le loro Applicazioni (GNSAGA) of the I\-sti\-tuto Naziona\-le di Alta Matematica (INdAM).

%\subsection*{Author contributions}

%This is an author contribution text. This is an author contribution text. This is an author contribution text. This is an author contribution text. This is an author contribution text. 

%\subsection*{Financial disclosure}

%None reported.

%\subsection*{Conflict of interest}

%The authors declare no potential conflict of interests.

%%%%%%%%%%%%%%%%%%%%%%%%%%%%%%%%%%%%%%%%%%%%%%%%%%%%%%%%%%%%%%%
%%%%%%%%%%%%%%%%%%%%%%%%%%%%%%%%%%%%%%%%%%%%%%%%%%%%%%%%%%%%%%%
%%%%%%%%%%%%%%%%%%%%%   BIBLIOGRAPHY
%%%%%%%%%%%%%%%%%%%%%%%%%%%%%%%%%%%%%%%%%%%%%%%%%%%%%%%%%%%%%%%
%%%%%%%%%%%%%%%%%%%%%%%%%%%%%%%%%%%%%%%%%%%%%%%%%%%%%%%%%%%%%%%
% ----------------------------------------------------------------
\bibliographystyle{amsplain}
\bibliography{bibliography}

\end{document}